\documentclass[11pt]{amsart}
\usepackage{amscd,amsmath,amssymb,amsfonts}
\usepackage{url,enumerate,color,hyperref}
\usepackage{dsfont}
\usepackage[cmtip, all]{xy}

\usepackage[top=30mm,right=30mm,bottom=30mm,left=20mm]{geometry}

\theoremstyle{plain}
\newtheorem{thm}{Theorem}
\newtheorem{lem}[thm]{Lemma}

\theoremstyle{definition}

\newtheorem{rmk}[thm]{Remark}

\numberwithin{equation}{section}

\newcommand{\Hom}{{\rm Hom}}

\newcommand{\Char}{{\rm char}}

\newcommand{\sym}{\mathrm{Sym}}

\newcommand{\Ind}{\text{Ind}}

\newcommand{\SL}{\mathrm{SL}}

\newcommand{\GL}{\mathrm{GL}}

\newcommand{\A}{{\mathbb A}}

{
\newcommand{\E}{{\mathbb E}}
\newcommand{\F}{{\mathbb F}}

\newcommand{\Z}{{\mathbb Z}}

\newcommand{\ZB}{{\mathbf Z}}

\newcommand{\reg}{\mathbf{reg}}

\begin{document}
\title{Symmetric powers}
\author{J\'anos Koll\'ar and Pham Huu Tiep}
\address{Department of Mathematics, Princeton University, Princeton, NJ 08544}
\email{kollar@math.princeton.edu}
\address{Department of Mathematics, Rutgers University, Piscataway, NJ 08854}
\email{tiep@math.rutgers.edu}
\thanks{The first author was partially supported by the NSF (grant DMS-1901855).}
\thanks{The second author gratefully acknowledges the support of the NSF (grant
DMS-2200850), the Simons Foundation, and the Joshua Barlaz Chair in Mathematics.}
\thanks{Part of this work was done when the second author visited Princeton University. 
It is a pleasure to thank Princeton University for generous hospitality and stimulating 
environment.}
\maketitle

Given a faithful (linear) action of a finite group $G$ on a finite-dimensional vector space $V$ over a field $\F$, it is 
a classical theme to study the {\it invariants} $\F[V]^G$, or {\it semi-invariants}, i.e. the one-dimensional $G$-submodules  
in $\F[V]$. One is particularly interested in the first nontrivial occurrence (in $\sym^m(V)$ with $m \geq 1$) 
of invariants, or semi-invariants see e.g. \cite{T}. 

More generally, one can ask whether (and when) a given irreducible 
$\F G$-module $W$ occurs, as a submodule or as a quotient, of $\sym^m(V)$. Motivated in particular by the study of automorphism
groups of rings of invariants, see e.g. \cite[Examples 10, 14]{K}, we give some answers to these questions.  

\section{Algebraically closed field case}
\begin{thm}\label{inf}
Let $\F$ be an algebraically closed field, $G$ any finite group, and $V = \F^n$ any finite-dimensional, faithful $\F G$-module. Let $W$ be 
any irreducible $\F G$-module $W$. Then there is some integer $1 \leq m \leq |G|$ (depending on $W$) such that the following 
statements hold.

\begin{enumerate}[\rm(i)]
\item $W$ is isomorphic to a submodule of $\sym^m(V)$. 
\item In fact, for any $k \in \Z_{\geq 0}$, $W$ is isomorphic to a submodule of $\sym^{m+k|G|}(V)$.
\item For any $k \in \Z_{\geq 0}$, $W$ is a quotient of $\sym^{m+k|G|}(V)$.
\end{enumerate}
\end{thm}

\begin{proof}
(i) Fixing a basis $(x_1, \ldots,x_n)$ of $V$, we can
identify $\sym^m(V)$ with the space of degree $m$ homogeneous polynomials in $n$ variables $x_1, \ldots,x_n$.

Let $Z:= \ZB(\GL(V)) \cap G$. Since $Z$ is isomorphic to a finite subgroup of $\F^\times$, $Z$ is cyclic. The statement is obvious in
the case $G = Z$, so we will assume $G > Z$ throughout.
Note that any element $g \in G \smallsetminus Z$ does not act as a scalar on $V$, and hence each of its 
eigenspaces in $V$ is a proper subspace of $V$, giving a proper closed subset of $V$ viewed as $\A^n_\F$. Since $\A^n_\F$ is an 
irreducible affine variety and $G$ is finite, the union of all these eigenspaces, when $g$ runs over $G \smallsetminus Z$, 
is a proper closed subset of $V$. 
It follows that we can find $0 \neq v \in V$ such that no element $g \in G \smallsetminus Z$ can fix the line $\langle v \rangle_\F$.

For any coset $\bar{h}= hZ$, note that the line $\langle h(v) \rangle_\F$ does not depend on the choice of the representative $h \in G$;
we will use $\bar{h}(v)$ to denote $h(v)$ for some choice of $h$.  Now we set
\begin{equation}\label{prod1}
  F(\bar{g}) := \prod_{\bar{h} \in G/Z \smallsetminus \{\bar{g}\}}\bar{h}(v) \in \sym^{N-1}(V),
\end{equation}  
where $N:= |G/Z| > 1$. By this definition, for any $g,h \in G$ and any $j \in \Z_{\geq 1}$, we see that $g$ sends  the line
$\langle F(\bar{h})^j\rangle_\F$ to $\langle F(\bar{g}\bar{h})^j\rangle_\F$. In particular, $G/Z$ permutes the $N$ lines 
$\langle F(\bar{g})^j\rangle_\F$, $\bar{g} \in G/Z$, regularly.

\smallskip
Next we show that, for any $j \in \Z_{\geq 1}$,  the $N$ polynomials $F(\bar{g})^j$, $\bar{g} \in G/Z$, are linearly independent.
Assume the contrary:
\begin{equation}\label{prod2}
  \sum_{\bar{g} \in G/Z}a_{\bar{g}}F(\bar{g})^j = 0
\end{equation}  
for some constants $a_{\bar{g}} \in \F$, and $a_{\bar{g}_0} \neq 0$. Applying $g_0^{-1}$ to this relation, we may assume that $a_{\bar{1}} \neq 0$.
For any $\bar{g} \neq \bar{1}$, the definition \eqref{prod1} implies that the polynomial $F(\bar{g})$ is divisible by $v$ in 
$k[x_1, \ldots ,x_n]$. As $a_{\bar{1}} \neq 0$, it follows from \eqref{prod2} that $F(\bar{1})^j$ is divisible by $v$ in 
$k[x_1, \ldots,x_n]$ which is a unique factorization domain. Using \eqref{prod1}, we see that there is some $h \in G \smallsetminus Z$ 
such that $h(v)$ is a multiple of $v$, contrary to the choice of $v$.

\smallskip
We have shown that the $\F G$-submodule 
$$A_j := \langle F(\bar{g})^j \mid \bar{g} \in G/Z \rangle_{\F}$$
of $\sym^{(N-1)j}(V)$ has dimension $N$ and is induced from the one-dimensional $\F Z$-module $\langle F(\bar{1})^j \rangle_\F$.
In fact,
$$A_j = \Ind^G_Z\bigl(\lambda^{(N-1)j}\bigr),$$
if $\lambda$ denotes the linear character via which $Z$ acts on $V$.

\smallskip
Clearly, $\lambda:Z \to \F^\times$ is injective; also, $Z \leq \ZB(G)$. Now we turn to $W$ and note that, since $W$ is irreducible,
$Z$ acts on $W$ via scalars by Schur's lemma. But $Z$ is cyclic, so there is some integer $0 \leq i \leq |Z|-1$ such that 
$Z$ acts on $W$ via the character $\lambda^i$. 
We also choose 
$$j:= |Z|-i \geq 1$$ 
and note that
\begin{equation}\label{prod3}
  Ni+(N-1)j = (N-1)|Z|+i \equiv i \pmod{|Z|}.
\end{equation}
Also consider    
$$B_i := \langle \prod_{\bar{h} \in G/Z}(\bar{h}(v)^i \rangle_{\F},$$
which is a one-dimensional $\F G$-submodule of $\sym^{Ni}(V)$. Then $Z$ acts on $B_i$ via the character $\lambda^{Ni}$. It follows
from \eqref{prod3} that $Z$ acts on the $\F G$-submodule $A_jB_i$ in $\sym^m(V)$ with character $\lambda^m=\lambda^i$, where
$$m:= Ni+(N-1)j = N|Z|-j < |G|.$$
Since 
$$A_jB_i = \Ind^G_Z\bigl(\lambda^{i}\bigr),$$
we now have 
$$0 \neq \Hom_Z(W|_Z,\lambda^i) \cong \Hom_G(W,A_jB_i)$$
by Frobenius reciprocity. As $W$ is irreducible, it follows that $W$ is submodule of $A_jB_i$ and hence of 
$\sym^m(V)$.

Similarly, 
\begin{equation}\label{quot1}
  0 \neq \Hom_Z(\lambda^i,W|_Z) \cong \Hom_G(A_jB_i,W),
\end{equation}  
which shows that $W$ is a quotient of $A_jB_i$.

\smallskip
(ii) We can argue as above, but replacing $A_jB_i$ by $A_jB_iC$, where 
$$C = \langle \prod_{h \in G}h(v)^k\rangle_\F,$$
which is a trivial $\F G$-submodule in $\sym^{k|G|}(V)$.

\smallskip
(iii) Under the extra assumption that $\Char(\F)=0$ or 
$\Char(\F) > 0$ is coprime to $|G|$, the $G$-module $\sym^{m+k|G|}(V)$ is semisimple. 
In this case, the statement follows from (i) and (ii).

Assume now that $\Char(\F)=p$ is a divisor of $|G|$. Since the finite subgroup $Z$ embeds in $\F^\times$, $Z$ is a cyclic group
of $p'$-order, and the regular $\F Z$-module $\reg_Z$ is semisimple, in fact isomorphic to the direct sum
$$\bigoplus^{|Z|-1}_{t=0}\lambda^t$$
(where we have identified one-dimensional modules with their Brauer characters). In turn, this implies that
the regular $\F G$-module
$$\reg_G = \Ind^G_1(1) \cong \Ind^G_Z(\Ind^Z_1(1)) = \Ind^G_Z(\reg_Z)$$  
is isomorphic to the direct sum
$$\bigoplus^{|Z|-1}_{t=0}\Ind^G_Z(\lambda^t).$$
Thus the submodule $A_jB_iC$ in (ii) is a direct summand of $\reg_G$, and hence projective and injective, see \cite[Lemma 45.3]{D}. 
Since $A_jB_iC$ is 
a submodule of $\sym^{m+k|G|}(V)$, it follows that it is also a quotient of $\sym^{m+k|G|}(V)$. But $A_jB_iC$ projects onto
$W$ by \eqref{quot1}, so $W$ is a quotient of $\sym^{m+k|G|}(V)$.
\end{proof}

\section{General case}

First we recall a well-known fact:

\begin{lem}\label{dim}
Let $G$ be a finite group, $\F$ any field, $U$ and $V$ any finite-dimensional $\F G$-modules. For any extension $\E/\F$, we have
$$\dim_\E\Hom_G(U \otimes_\F \E,V \otimes_\F \E) = \dim_\F \Hom_G(U,V).$$
\end{lem}

\begin{proof}
Note that 
$$\Hom(U \otimes_\F \E,V \otimes_\F \E) \cong (V^* \otimes_\F U) \otimes_\F \E$$
as $\E G$-modules, and hence the $\E$-vector space $\Hom_G(U \otimes_\F \E,V \otimes_\F \E)$ is isomorphic to
the $G$-fixed point subspace $\tilde X$ of $(V^* \otimes_\F U) \otimes_\F \E$. Similarly, the $\F$-vector space 
$\Hom_G(U,V)$ is isomorphic to the $G$-fixed point subspace $X$ of $V^* \otimes_\F U$. Note that 
$$\dim_\F X = \dim_\E\tilde X.$$
(Indeed, fix a basis $(e_1, \ldots, e_n)$ of $V^* \otimes U$, and let $A(g)$ be the matrix of the action of $g \in G$ relative to this
basis. List the elements in $G$ as $(g_1, \ldots, g_N)$ for $N:=|G|$, and set 
$$A:=[A(g_1);A(g_2); \ldots; A(g_N)] \in \mathrm{Mat}_{nN,n}(\F).$$
Then $X$ is the solution space of the linear system $A\mathbf{x}=\mathbf{0}$ over $\F$, and    
$\tilde X$ is the solution space of the linear system $A\mathbf{x}=\mathbf{0}$ over $\E$.)
Hence the statement follows.
\end{proof}

\begin{thm}\label{main1}
Let $\F$ be any field, $G$ any finite group, and $V = \F^n$ any finite-dimensional, faithful $\F G$-module. Let $W$ be 
any irreducible $\F G$-module $W$. Then there is some integer $1 \leq m \leq |G|$ (depending on $W$) such that the following 
statements hold.
\begin{enumerate}[\rm(i)]
\item $W$ is isomorphic to a submodule of $\sym^m(V)$. 
\item $W$ is isomorphic to a quotient of $\sym^m(V)$. 
\item In fact, for any $k \in \Z_{\geq 0}$, $W$ is isomorphic to a submodule of $\sym^{m+k|G|}(V)$, and 
to a quotient of $\sym^{m+k|G|}(V)$.
\end{enumerate}
\end{thm}

\begin{proof}
(i) Set $\tilde V := V \otimes_\F\overline{\F}$ and $\tilde W:= W \otimes_\F \overline{\F}$. Let $W_0$ be a simple quotient of 
the $\overline{\F}G$-module $\tilde W$.
We may view $G < \GL(V) < \GL(\tilde V)$ and apply Theorem  
\ref{inf}(i) to $(\overline{\F},\tilde V,W_0)$. Hence for the integer $m$ in Theorem \ref{inf}(i), we have that 
$W_0$ is a submodule of 
$$\sym^m(\tilde V) \cong U \otimes_\F \overline{\F},$$ 
where 
$$U:= \sym^m(V).$$
Thus 
$$0 \neq \Hom_G(W_0,U \otimes_\F \overline{\F}).$$ 
Since $W_0$ is a quotient of $\tilde W$, it follows that
$$0 \neq \Hom_G(\tilde W,U \otimes_\F \overline{\F}).$$ 
Applying Lemma \ref{dim}, we obtain
$$\Hom_G(W,U) \neq 0.$$    
As $W$ is irreducible, it follows that $W$ embeds in $\sym^m(V)$. 

\smallskip
(ii) Keep the notation of (i), but take $W_0$ be a simple submodule of $\tilde W$. By Theorem  
\ref{inf}(iii) applied to $(\overline{\F},\tilde V,W_0)$, and for the integer $m$ in Theorem \ref{inf}(i), we have that 
$W_0$ is a quotient of 
$$\sym^m(\tilde V) \cong U \otimes_\F \overline{\F},$$ 
where $U:= \sym^m(V)$.
Thus 
$$0 \neq \Hom_G(U \otimes_\F \overline{\F},W_0).$$ 
Since $W_0$ is a submodule of $\tilde W$, it follows that
$$0 \neq \Hom_G(U \otimes_\F \overline{\F},\tilde W).$$ 
Applying Lemma \ref{dim}, we obtain
$$\Hom_G(U,W) \neq 0.$$    
As $W$ is irreducible, it follows that $W$ is quotient of $\sym^m(V)$. 

\smallskip
(iii) The same arguments as in (i), (ii) apply to $m+k|G|$ for any $k \in \Z_{\geq 0}$.
\end{proof}

\begin{rmk}
\begin{enumerate}[\rm(i)]
\item Let $V = \F = \overline{\F}$ and $G \leq \GL(V) \cong \GL_1(\F)$ be a finite cyclic subgroup. Then, to cover all irreducible 
$\F G$-modules $W$, we need to use all $\sym^m(V)$ with $m$ up to $|G|-1$.
\item Let $\Char(\F) = p > 0$, and consider the natural $2$-dimensional module $V = \F^2$ of $G \cong \SL_2(p)$. A full set of
non-isomorphic irreducible $\F G$-modules is given by $\sym^m(V)$, $0 \leq m \leq p-1$. Thus in Theorem \ref{main1}, 
we need to use $m$ up to $\approx |G|^{1/3}$.
\item Let $\F = \overline{\F}$ be of characteristic $p > 0$, $f \geq 2$ any integer, and consider the natural $2$-dimensional module 
$V = \F^2$ of $G \cong \SL_2(p^f)$. For any integer $0 \leq m \leq p^{f-1}-1$, the module $\sym^m(V)$ has highest weight $m\varpi_1$, 
where as the $(p^{f-1})^{\mathrm {th}}$ Frobenius twist $V^{(p^{f-1})}$ of $V$ has highest weight $p^{f-1}\varpi_1$, and is irreducible
of dimension $2$. So again in Theorem \ref{main1} we have to use $m$ up to $\approx |G|^{1/3}$.
\item In general, $\dim(W)$ can grow unboundedly compared to $\dim(V)$. Hence, the $m$ in Theorem \ref{main1} cannot be
bounded in terms of $\dim(V)$, unlike the situation in \cite{T}, cf. the examples in (i)--(iii). 
In fact the example in (i) shows that $m$ cannot be bounded in terms of
both $\dim(V)$ and $\dim(W)$, whereas (ii) shows $m$ cannot be bounded in terms of
$\dim(V)$ and $|\ZB(G)|$. Furthermore, the example in (iii) shows that $m$ cannot be bounded in terms of
$\dim(V)$, $\dim(W)$, $|\ZB(G)|$, and $p=\Char(\F)$. It looks plausible 
that $m$ can be bounded only in terms of $|G|$.
\end{enumerate}
\end{rmk}


\begin{thebibliography}{A}
\bibitem{D}
  L. Dornhoff, `{\it Group Representation Theory}', Dekker, New York, 1971.
\bibitem{K}
J. Koll\'ar, Automorphisms of rings of invariants, \url{arXiv:2212.03772v1}
\bibitem{T}
  Pham Huu Tiep, The $\alpha$-invariant and Thompson's conjecture, {\it Forum of Math. Pi} {\bf 4} (2016), e5, 28 pages. 
\end{thebibliography}
\end{document}